\documentclass[reqno]{amsart}
\usepackage{mathptmx}
\usepackage{pxfonts}
\usepackage{amssymb}
\usepackage{amsfonts}
\usepackage{amsmath}
\usepackage{graphicx}
\usepackage{shadow}
\usepackage{color}
\usepackage[all]{xy}
\usepackage[pagebackref]{hyperref}
\newtheorem{thm}{Theorem}[section]
\newtheorem{corollary}[thm]{Corollary}
\newtheorem{lemma}[thm]{Lemma}
\newtheorem{proposition}[thm]{Proposition}

\theoremstyle{definition}
\newtheorem{definition}[thm]{Definition}
\newtheorem{example}[thm]{Example}

\theoremstyle{remark}

\def\1{{\rm (1)}}
\def\2{{\rm (2)}}
\def\3{{\rm (3)}}
\def\4{{\rm (4)}}
\def\5{{\rm (5)}}

\makeatletter
  \newcounter{xenumi}
  
\makeatother

\begin{document}

\title[Finitely Star Regular Domains]{Finitely Star Regular Domains $^{(\star)}$}
\thanks{$^{(\star)}$ Supported by KFUPM under DSR Research Grant \#: RG161001.}


\author[A. Mimouni]{A. Mimouni}
\address{Department of Mathematics, KFUPM, Dhahran 31261, KSA}
\email{amimouni@kfupm.edu.sa}

\date{\today}



\begin{abstract}
 Let $R$ be an integral domain, $Star(R)$ the set of all star operations on $R$ and $StarFC(R)$ the set of all star operations of finite type on $R$. Then $R$ is said to be star regular if $|Star(T)|\leq |Star(R)|$ for every overring $T$ of $R$.
In this paper we introduce the notion of finitely star regular domain as an integral domain $R$ such that $|StarFC(T)|\leq |StarFC(R)|$ for each overring $T$ of $R$. First, we show that the notions of star regular and finitely star regular domains are completely different and do not imply each other. Next, we extend/generalize well-known results on star regularity in Noetherian and Pr\"ufer contexts to  finitely star regularity. Also we handle the finite star regular domains issued from classical pullback constructions to  construct finitely star regular domains that are not star regular and enriches the literature with a such class of domains.
\end{abstract}
\maketitle

\section{Introduction}

\noindent Let $R$ be an integral domain with quotient field $L$, $F(R)$ the set of nonzero fractional ideals of $R$ and $f(R)$ the set of nonzero finitely generated fractional ideals of $R$.
A mapping $*: F(R)\rightarrow F(R)$,
$E\mapsto E^*$, is called a {star operation} on $R$ if the
following conditions hold for all $a\in L\setminus \{0\}$ and $E,
F\in F(R)$:
\begin{enumerate}
\item[(I)] $(aE)^*=aE^*$; \item[(II)] $E\subseteq E^*$;\, if
$E\subseteq F$, then $E^*\subseteq F^*$;\, and \item[(III)]
$(E^*)^*=E^*$.
\end{enumerate}
The simplest star operations are the $d$-operation defined by $E^{d}=E$ for every $E\in F(R)$, and the $v$-operation defined by $E^{v}=(E^{-1})^{-1}$ (where $E^{-1}=(R:E)=\{x\in L|xE\subseteq R\}$) for every $E\in F(R)$. A star operation $*$ is said to be of finite type (or of finite character) if for each nonzero (fractional) ideal $E$ of $R$, $E^{*}=\bigcup F^{*}$ where the union is taken over all nonzero finitely generated subideals $F$ of $E$. Also a star operation is stable if $(E\cap F)^{*}=E^{*}\cap F^{*}$ for each $E, F\in F(R)$. To any star operation $*$ on $R$, we associate a star operation of finite type $*_{f}$ and a stable star operation of finite type $*_{w}$ by setting respectively $E^{*_{f}}=\bigcup \{F^{*}| F\in f(R), F\subseteq E\}$  and $E^{*_{w}}=\bigcup\{(E:F)|F\in f(R), F^{*}=R\}$. Notice that $v_{f}=t$ and $t_{w}=w$. For star operations $*$ and $*'$ on $R$, $*\leq *'$ provided that $E^{*}\subseteq E^{*'}$ for every $E\in F(R)$. Clearly $d\leq w\leq t\leq v$ and for every star operation $*$ on $R$, $d\leq *\leq v$ and $d\leq *_{f}\leq t$. We denote by $\text{Star}(R)$ the set of all star operations on $R$ and $StarFC(R)$ the set of all star operations of finite type on $R$.\\ 
Recently, motivated by well-known characterizations of integrally closed and Noetherian divisorial domains \cite{Hei, m1}, the author of this paper, together with E. Houston and M. H. Park, started a long and deep study of some ring-theoretic properties of integral domains having only finitely many star operations
in different contexts of integral domains. Namely, complete characterizations are given in the cases of local Notherian domains with infinite residue field and integrally closed domains see \cite{hmp1, hmp2, hmp3, hmp4, hmp5} (see also \cite{C, hp, hlp, Par1}, and \cite{Sp1, Sp2}).\\
In \cite{hmp4}, the authors studied, for a Noetherian domain $R$, how $|Star(R)|$ affects $|Star(T)|$ for each proper overing $T$ of $R$ with the emphasis on the case where $Star(R)$ is finite. They introduced the notion of a star regular domain as a domain $R$ such that $|Star(T)|\leq |Star(R)|$ for each overring $T$ of $R$. Notice that a Noetherian domain $R$ (which is not a field) with finitely many star operations has Krull dimension one. The authors constructed a Noetherian domain $R$ with $|Star(R)|=1$ (equivalently, $R$ is a divisorial domain), but having an overring $T$ with $|Star(T)|=\infty$. Next, they showed that for a one-dimensional Noetherian domain $R$, if $R$ is locally star regular, then it is star regular, and the converse holds if $Star(R)$ is finte. They conjectured that ``if $R$ is a local Noetherian domain with $1< |Star(R)|< \infty$, then $R$ is star regular'', and proved that this conjecture holds if $R$ has infinite residue field. They also considered the question of whether finiteness of $Star(T)$ for each proper overring of a Noetherian domain $R$ implies finiteness of $Star(R)$, and showed that this occurred when R is non-local.\\
In \cite{hmp5}, the authors investigated star regular domains in the context of Pr\"ufer domains. They showed that star regularity for Pr\"ufer domains with only finitely many star operations reduces to star regularity of Pr\"ufer domains $R$ possessing a nonzero
prime ideal $P$ contained in the Jacobson radical of $R$ such that $Spec(R/P)$ is finite (\cite[Theorem 3.1]{hmp5}).  More precisely, they proved that if $R$ is a semi-local Pr\"ufer domain with more than one maximal ideal such
that $R/P$ is strongly discrete (where $P$ is the largest prime ideal contained in the Jacobson radical of $R$), then $R$ is star regular (\cite[Theorem 3.11]{hmp5}).\\
In \cite{KadMi}, the authors investigated some ring-theoretic properties of certain classes of integral domains with only finitely many star operations of finite type. Several generalizations/analogues of well-known results on integral domains with finitely many star operations were extended to integral domains with finitely many star operations of finite type. Namely in Noetherian-like settings such us Mori domains, pullback constructions and more.\\

The purpose of this paper is to introduce and study the notion of finitely star regular domains, that is, integral domains $R$ such that $|StarFC(T)|\leq |StarFC(R)|$ for each overring $T$ of $R$. The class of finitely star regular domains includes the class of Pr\"ufer domains (since $|StarFC(R)|=1$ for every Pr\"ufer domain $R$), and coincides with the class of star domains in the Neotherian context. Motivated by the fact that Pr\"ufer domains are finitely star regular domains, but not always star regular, in section 2, we start by showing that the two notions of star and finitely star regular domains are completely different and do not imply each other, (see Example~\ref{fsr-not-sr.3}, Example~\ref{fsr-not-sr.2}, and Example~\ref{fsr-not-sr.1}). Next, we extend \cite[Theorem 2.3]{hmp2} to a one-dimensional domain of finite character (i.e., every nonzero non-unit element is contained in a finitely many maximal ideals). That is, for a one-dimensional domain $R$ of finite character such that $StarFC(R)$ is finite,  $|StarFC(R)|=\displaystyle\prod_{M\in Max(R)}|StarFC(R_{M})|$ (Theorem~\ref{Th2}). The second main theorem asserts that  if $R$ is a one-dimensional quasi-Pr\"ufer domain (i.e. its integral closure is a Pr\"ufer domain) such that $StarFC(R)$ is finite, then $R$ is finitely star regular if and only if $R_{M}$ is finitely star regular for every maximal ideal $M$ of $R$ (Theorem~\ref{Th3}). The third section deals with certain pullbacks that are finitely star regular. The main result asserts that a large class of integral domains related to valuation domains, namely the class of pseudo-valuation domains ($PVD$ for short), are always finitely star regular, but not star regular in general. Next, we deal with the classical pullbacks issued from valuation domains in order to enriches the literature with such a class of integral domains.\\

Throughout, $R$ denotes an integral domain (which is not a field), $R'$ its integral closure and $\overline{R}$ its complete integral closure. The set of all maximal ideals of $R$ is denoted by $Max(R)$, and if $S$ is an overring of $R$, $[R, S]$ denotes the set of all intermediate rings between $R$ and $S$ (i.e. rings $T$ such that $R\subseteq T\subseteq S$).
\section{General Results}

Recall that an integral domain $R$ is said to be a star regular domain if $|Star(T)|\leq |Star(R)|$ for every overring of $R$. Next, we introduce the notion of finitely star regular domains.

\begin{definition} Let $R$ be an integral domain. We say that $R$ is finitely star regular if $|StarFC(T)|\leq |StarFC(R)|$ for each overring $T$ of $R$.
\end{definition}
The next proposition deals with the class of integrally closed domains and shows that it is an important class of finitely star regular domains that are not necessarily star regular. 
\begin{proposition} Let $R$ be an integrally closed domain. Then $R$ is finitely star regular.
\end{proposition}

\begin{proof} Assume that $R$ is integrally closed. If $|StarFC(R)|=\infty$, we are done. Assume that $StarFC(R)$ is finite. Then by \cite[Theorem 3.1]{hmp3}, $R$ is a Pr\"ufer domain. Thus every overring $T$ of $R$ is Pr\"ufer and so $|StarFC(T)|=|StarFC(R)|=1$, as desired.
\end{proof}
Our next three examples show that the notions of star regular domain and finitely star regular domain do not implies each other. The first one is an example of a finitely star regular domain $R$ with $Star(T)$ finite for every overring $T$ of $R$, but which is not star regular (as it has an overring $T$ with $|Star(R)|<|Star(T)|$). The second one is an example of a star regular domain which is not finitely star regular with $StarFC(R)$ finite but $Star(R)$ is infinite. In the third example, $R$ is finitely star regular but not star regular with $Star(R)$ finite and having an overring $T$ with $|Star(T)|=\infty$. Notice that if $V$ is a valuation domain, then $|Star(V)|\leq 2$.
\begin{example}\label{fsr-not-sr.3} Let $V$ be a valuation domain with a principal maximal ideal $M$ and suppose that $V$ has a nonzero non-maximal ideal $P$ such that $P=P^{2}$. Clearly for every overring $T$ of $V$, $|StarFC(T)|=|StarFC(V)|=1$. Thus $V$ is a finitely star regular domain. However,  $|Star(V_{P})|=2$ while $|Star(V)|=1$. Hence $V$ is not star regular.
\end{example}
 \begin{example}\label{fsr-not-sr.2} The following is an example of a star regular domain which is not finitely star regular. Let $k$ be an infinite field, $X$ and $Y$ indeterminates over $k$ and set $D=k[X^{3}, X^{7}]$ and $D_{1}=k[X^{3}, X^{7}, X^{8}]$. By \cite[Theorem 2.2]{M2}, $D$ is a divisorial Noetherian domain and so $|Star(D)|=|StarFC(D)|=1$, while $D_{1}$ is a Noetherian overring of $D$ with $|Star(D_{1})|=|StarFC(D_{1})|=\infty$. Indeed, let $M_{1}=(X^{3}, X^{7}, X^{8})$. Then $M_{1}$ is a maximal ideal of $D_{1}$ with $(D_{1}:M_{1})=k[X^{3}, X^{4}, X^{5}]$ and so $(D_{1})_{M_{1}}$ is a Noetherian local domain satisfying the conditions of \cite[Theorem 3.9]{hmp2}. Thus $|StarFC((D_{1})_{M_{1}})|=|Star((D_{1})_{M_{1}})|=\infty$. By \cite[Theorem 2.3]{hmp2}, $|StarFC(D_{1})|=|Star(D_{1})|=\infty$. Now set $V=k(X)[[Y]]=k(X)+Yk(X)[[Y]]$, $R=D+Yk(X)[[Y]]=D+M$ and $T=D_{1}+M$. By \cite[Theorem 4.4]{KadMi}, $|StarFC(R)|=|StarFC(D)|=1$, but $|StarFC(T)|=|StarFC(D_{1})|=\infty$. Thus $R$ is not finitely star regular. However, $|Star(R)|=\infty$ and so $R$ is star regular. Indeed,  for every maximal ideal $P$ of $k[X]$, set $S_{P}=k[X]_{P}+M$. Then $S_{P}$ is a fractional overring of $R$, and $(R:S_{P})=M$, so that $(S_{P})_{v}=(R:M)=V$. Now, by \cite[Proposition 2.7]{hmp1}, $*_{P}=\delta(d_{S_{P}}, v)$ defined by $E^{*_{P}}=ES_{P}\cap E_{v}$ is a star operation on $R$. Moreover, for $P\not=Q$ maximal ideals of $k[X]$, $(S_{P})^{*_{P}}= S_{P}\cap V=S_{P}$ and $(S_{P})^{*_{Q}}=S_{Q}S_{P}\cap V=V\cap V=V$. Hence $*_{P}\not=*_{Q}$. Since $k[X]$ has infinitely many maximal ideals that induce different star operations on $R$, $R$ has infinitely many star operations.
\end{example} 

The following is an example of a non star regular Pr\"ufer domain $R$ with only finitely many star operations having an overring $T$ with $|Star(T)|=\infty$. The Example is given in \cite[Example 2.4]{hmp5}, and for the convenience of the reader, we include it here.

\begin{example}\label{fsr-not-sr.1} The following is an example of a finitely star regular which is not star regular.
Let $A$ be the direct sum of countably infinitely many copies of
$G$, where $G$ is the totally ordered group $\mathbb R\bigoplus
\mathbb Z$ with lexicographic order (i.e., for $(a, b), (a',
b')\in \mathbb R\bigoplus \mathbb Z$, $(a, b)\geq (a', b')$ if
$a>a'$, or $a=a'$ and $b\geq b'$).  For $(a_i)_{i=1}^\infty\in A$,
define $(a_i)\geq 0$ if $a_i\geq 0$ for all $i\geq 1$. Then $A$ is
a lattice ordered group. According to Jaffard and Ohm \cite{J, O}, every lattice ordered group is a group of divisibility of a B\'ezout
domain. Thus there is a B\'ezout domain $R$  with $K^*/U(R)\cong A$, where $U(R)$ is the set of units of $R$. Since
$A$ is the weak direct product of $G$'s, $R$ has finite character. It follows easily from \cite[Theorem 3.2]{L} and \cite[Lemma
3.4]{hlp} that $R$ is an $h$-local Pr\"{u}fer domain with
$\text{dim}\, R=2$. Thus each overring $T$ of $R$ is a Pru\"ufer domain and so $|StarFC(T)|=|StarFC(R)|=1$. Hence $R$ is finitely star regular. Since for each maximal ideal $M$ of $R$, $R_M$
has value group $G$, $MD_M$ is principal. Since $R$ has finite
character, each $M$ is invertible. Hence $R$ is a divisorial
domain \cite[Theorem 5.1]{Hei}.  Let $M$ be a maximal ideal of
$R$, and let $P$ be the nonzero prime contained in $M$. Since the
value group of $R_P$ is equal to $\mathbb R$, $PR_P$ is not
principal. Hence $P=P^2$, and so $R_M$ is not strongly discrete.
By \cite[Theorem 2.3]{hmp5}, there is an overring $T$ of $R$ with
$|Star(T)|=\infty$ (and we may take $T=\bigcap R_P$, where the
intersection is taken over the height-one primes of $R$).\\
Notice that, for each positive integer $n$, if in the definition of $A$, we replace $n$ of the $G$ by $\mathbb R \bigoplus \mathbb R$, then the resulting $R$ satisfies $|Star(R)|=2^n$ by \cite[Theorem 3.1]{hmp1}, and $R$ has an overring $T$ with
 $|Star(T)|=\infty$ (\cite[Remark 2.5]{hmp5}).
\end{example}
The next two lemmas  are crucial. The first one is \cite[Proposition 2.7]{hmp1} and we shall use it whenever we consider a proper overring $T$ of the base ring $R$ and a star operation on $T$. The second one is a direct combination of \cite[Proposition 2.4 and Proposition 4.6]{hmp1} and we shall use it whenever we consider the particular overring $T=R_{M}$ for some maximal ideal $M$.
\begin{lemma}\label{L1}(\cite[Proposition 2.7]{hmp1}) Let $R$ be an integral domain, $T$
an overring of $R$, $*$ a star operation on $R$ and $*'$ a star
operation on $T$. Then the map $\delta(*', *):F(R)\rightarrow
F(R)$, $E\mapsto E^{\delta(*', *)}:=(ET)^{*'}\cap E^{*}$, defines a
star operation on $R$. Moreover, if $*$ and $*'$ are of finite
type, then so is $\delta(*', *)$.
\end{lemma}

\begin{lemma}\label{L2} Let $R$ be an integral domain, $M$ a maximal ideal of $R$ and $*\in \text{StarFC}(R)$. Then $*(M)$ defined on $R_{M}$ by $(AR_{M})^{*(M)}=A^{*}R_{M}$ for every $A\in f(R)$ is a star operation on $R_{M}$. Moreover, if $R$ is $v$-coherent, then every star operation on $R_{M}$ is of this form, that is,  if $*'\in \text{StarFC}(R_{M})$, then $*'=*(M)$ for some $*\in \text{StarFC}(R)$.
\end{lemma}

\begin{thm}\label{Th1} Let $R$ be an integral domain such that $(R:\overline{R})=M$ is a maximal ideal of $R$ and suppose that $\overline{R}$ is a $PID$. Then for every $T\in[R, \overline{R}]$, $|Star(T)|\leq |Star(R)|$ and $|StarFC(T)|\leq |StarFC(R)|$.
\end{thm}

\begin{proof} First notice that $M^{-1}=(M:M)=\overline{R}$ and so $M$ is a divisorial ideal of $R$. Let $T\in [R, \overline{R}]$. Since $R$ is a $PID$, $|StarFC(\overline{R})|=1\leq |Star FC(R)|$, so without loss of generality we may assume that $R\subsetneqq T\subsetneqq \overline{R}$. Thus $(R:T)=M$ and so $M$ is an ideal of $T$ and $T_{v}=M^{-1}=\overline{R}$. Notice that if $S$ is a fractional overring of $T$, then $S\subseteq \overline{S}=\overline{T}=\overline{R}=M^{-1}$. Since 
$(M^{-1})^{*}$ is a fractional overring of $T$, then $(M^{-1})^{*}=M^{-1}$ for every $*\in Star(T)$. Thus for every nonzero fractional ideal $I$ of $T$ with $IM^{-1}=xM^{-1}$, $I^{*}\subseteq (IM^{-1})^{*}=(xM^{-1})^{*}=xM^{-1}=IM^{-1}$. Now let $*_{1}\not=*_{2}$ be star operations (resp. star operations of finite type) on $T$ and let $A$ be a fractional ideal (resp. a finitely generated fractional ideal) of $T$ such that $A^{*_{1}}\not=A^{*_{2}}$. Without loss of generality, we may assume that $A\subseteq M$ (for if $0\not=d\in T$ is such that $dA\subseteq T$ and $0\not=m\in M$, $mdA\subseteq mT\subseteq M$ and  $(mdA)^{*_{1}}=mdA^{*_{1}}\not=mdA^{*_{2}}=(mdA)^{*_{2}}$). Notice that $A$ is not a divisorial ideal of $T$ and so $A$ is not an invertible ideal of $T$. Thus $A(R:A)\not=R$. Now set $AM^{-1}=aM^{-1}$. Then
$AM=AMM^{-1}=aMM^{-1}=aM$ and $(M:A)=(M_{v}:A)=((R:M^{-1}):A)=(R:AM^{-1})=(R:aM^{-1})=a^{-1}M_{v}=a^{-1}M$. Thus $A(M:A)=a^{-1}AM=M$. Hence $M=A(M:A)\subseteq A(R:A)\subsetneqq R$ and by maximality of $M$, $M=A(M:A)=A(R:A)$. Thus $(R:A)=(M:A)=a^{-1}M$ and so $A_{v}=aM^{-1}=AM^{-1}$.
Now, by Lemma~\ref{L1},  $A^{\delta(*_{1}, v)}=A^{*_{1}}\cap A_{v}=A^{*_{1}}\cap AM^{-1}=A^{*_{1}}$ and $A^{\delta(*_{2}, v)}=A^{*_{2}}\cap A_{v}=A^{*_{2}}\cap AM^{-1}=A^{*_{2}}$. Thus $\delta(*_{1}, v)\not=\delta(*_{2}, v)$.\\
For star operations of finite type $*_{1}\not=*_{2}$ and $A$ a finitely generated ideal of $T$ ($A\subseteq M$) with $A^{*_{1}}\not=A^{*_{2}}$, let $B$ be a finitely generated ideal of $R$ such that $BT=A$ (for $A=\displaystyle\sum_{i=1}^{i=n}a_{i}T$, just take $B=\displaystyle\sum_{i=1}^{i=n}a_{i}R$). If since $B(R:B)\subseteq A(T:A)\subsetneqq T$, $B(R:B)\subsetneqq R$. Moreover, $BM^{-1}=AM^{-1}=aM^{-1}$. Thus $(M:B)=(M_{v}:B)=((R:M^{-1}):B)=(R:BM^{-1})=(R:aM^{-1})=a^{-1}M_{v}=a^{-1}M$. Hence $B(M:B)=a^{-1}BM=M$ and so $M=B(M:B)\subseteq B(R:B)\subsetneqq R$. Thus $M=M(B:M)=M(R:M)$ and so $(R:B)=(M:B)=a^{-1}M$. Hence $B_{t}=B_{v}=aM^{-1}=BM^{-1}=AM^{-1}$. Finally,
$B^{\delta(*_{1}, t)}=(BT)^{*_{1}}\cap B_{t}=A^{*_{1}}\cap AM^{-1}=A^{*_{1}}$ and $B^{\delta(*_{2}, t)}=(BT)^{*_{2}}\cap B_{t}=A^{*_{2}}\cap AM^{-1}=A^{*_{2}}$. It follows that $\delta(*_{1}, t)\not=\delta(*_{2}, t)$. In both cases, the map $\delta(-,v): Star(T)\longrightarrow Star(R)$, $* \longmapsto \delta(*, v)$ (resp. $\delta(-,t): StarFC(T)\longrightarrow StarFC(R)$, $*\longmapsto \delta(*, t)$) is one-to-one. Hence $|Star(T)|\leq |Star(R)|$ (resp. $|StarFC(T)|\leq |StarFC(R)|$), as desired.
\end{proof}
Recall that an integral domain $R$ is said to be conducive if the conductor $(R:T)\not=0$ for every overring $T$ of $R$ with $T\subsetneqq qf(R)$, equivalently $(R:V)\not=0$ for some valuation overring $V$ of $R$. The next corollary generalizes \cite[Proposition 1.10]{hmp4}.
\begin{corollary} (\cite[Proposition 1.10]{hmp4}) Let $(R, M)$ be a local Noetherian domain such that $M^{-1}$ is a valuation domain. Then $R$ is star regular.
\end{corollary}

\begin{proof} Since $M^{-1}$ is a valuation domain, $R$ is a conducive domain and clearly $M^{-1}=(M:M)=\overline{R}$. Thus $[R, L]=[R, \overline{R}]\cup\{L\}$. By Theorem~\ref{Th1}, $R$ is star regular.
\end{proof}
It is well-known that if $R$ is a Noetherian domain (which is not a field) with $Star(R)=StarFC(R)$ finite, then $R$ has Krull dimension $dim(R)=1$ (\cite[Theorem 2.1]{hmp2}). Therefore $R$ is of finite character, that is, each nonzero nonunit element is contained in only finitely many maximal ideals. In \cite[Theorem 2.3]{hmp2}, it was proved that for a Noetherian domain $R$, $|StarFC(R)|=|Star(R)|=\displaystyle\prod_{M\in Max(R)}|Star(R_{M})|$. Our first theorem generalizes this result to one-dimensional domains of finite character. 
\begin{thm}\label{Th2} Let $R$ be a one-dimensional domain of finite character such that $StarFC(R)$ is finite. Then $|StarFC(R)|=\displaystyle\prod_{M\in Max(R)}|StarFC(R_{M})|$.
\end{thm}

\begin{proof} First set $\mathcal{M}_{1}:=\{M\in Max(R):|StarFC(R_{M}|=1\}$ and $\mathcal{M}_{\geq 2}:=\{M\in Max(R):|StarFC(R_{M}|\geq 2\}$. Clearly $Max(R)$ is the disjoint union of $\mathcal{M}_{1}$ and  $\mathcal{M}_{\geq 2}$. Now if $\mathcal{M}_{\geq 2}=\emptyset$, then $|StacFC(R_{M})|=1$ and so $d_{M}=t({M})$ for every $M\in Max(R)$. Let $*\in StarFC(R)$ and let $I$ be a finitely generated ideal of $R$. Then for every $M\in Max(R)$, $I_{t}R_{M}=(IR_{M})^{t(M)}=(IR_{M})^{d_{M}}=IR_{M}$. Hence $I_{t}=I$ and so $t=d$. Thus $StarFc(R)=\{d\}$ and therefore $|StarFC(R)|=1=\displaystyle\prod_{M\in Max(R)}|StarFC(R_{M})|$. Assume that $\mathcal{M}_{\geq 2}\not=\emptyset$.\\
{\bf Claim}: $\mathcal{M}_{\geq 2}$ is finite. By way of contradiction suppose that $\mathcal{M}_{\geq 2}$ is infinite. Then for every positive integer $n$ consider a subset $\{M_{1}, \dots, M_{n}\}\subseteq \mathcal{M}_{\geq 2}$, and consider the map $\phi: \displaystyle\prod_{i=1}^{n}StarFC(R_{M_{i}})\longrightarrow StarFC(R)$, $\star=(*_{i})_{i=1}^{n}\mapsto \phi(\star)=\star_{\phi}$ where $\star_{\phi}$ is defined by $A^{\star_{\phi}}=\displaystyle\bigcap_{i=1}^{n}(AR_{M_{i}})^{*_{i}}\cap\displaystyle\bigcap_{M\in Max(R): M\not=M_{i}}(AR_{M})$. By \cite[Theorem 2]{A}, $\star_{\phi}\in StarFC(R)$ and so $\phi$ is well-defined. Now, let $\star=(*_{i})_{i=1}^{n}\not =(*'_{i})_{i=1}^{n}=\star'$. Then $*_{j}\not=*'_{j}$ for some $j\in \{1, \dots, n\}$. Let $E$ be a finitely generated integral ideal of $R_{M_{j}}$ such that $E^{*_{j}}\not=E^{*'_{j}}$ and let $A=E\cap R$. Since $dim(R_{M_{j}})=1$, $E$ is $M_{j}R_{M_{j}}$-primary and so $A$ is $M_{j}$-primary. Hence $Max(R, A)=\{M_{j}\}$ and so $AR_{M}=R_{M}$ for every $M\in Max(R)-\{M_{j}\}$. So $A^{\star_{\phi}}=(AR_{M_{j}})^{*_{j}}\cap\displaystyle\bigcap_{M\in Max(R): M\not=M_{j}}R_{M}=E^{*_{j}}\cap\displaystyle\bigcap_{M\in Max(R): M\not=M_{j}}R_{M}$. Thus $A^{\star_{\phi}}\cap R_{M_{j}}=E^{*_{j}}\cap R$. Similarly $A^{\star'_{\phi}}=E^{*'_{j}}\cap R$. So if $A^{\star_{\phi}}=A^{\star'_{\phi}}$, then $E^{*_{j}}\cap R=A^{\star_{\phi}}\cap R_{M_{j}}=A^{\star'_{\phi}}\cap R_{M_{j}}=E^{*'_{j}}\cap R$. Hence $E^{*{j}}= (E^{*_{j}}\cap R)R_{M_{j}}=(E^{*'_{j}}\cap R)R_{M_{j}}=E^{*'_{j}}$, which is absurd. Hence $A^{\star_{\phi}}\not=A^{\star'_{\phi}}$ and so $\phi(\star)\not=\phi(\star')$. Thus $\phi$ is a one-to-one and therefore $\displaystyle\prod_{i=1}^{n}|StarFC(R_{M_{i}})|\leq |StarFC(R)|$, for every positive integer $n$. So $|StarFC(R)|=\infty$, which is a contradiction, completing the proof of the claim.\\
Now assume that $\mathcal{M}_{\geq 2}=\{M_{1}, \dots, M_{r}\}$. Let $\varphi:StarFC(R)\longrightarrow \prod_{M\in Max(R)}StarFC(R_{M})$ defined by $\varphi(*)=(*(M))_{M\in Max(R)}$. Clearly $\varphi$ is one-to-one and so\\
$|StarFC(R)|\leq \displaystyle\prod_{M\in Max(R)}|StarFC(R_{M})|=\displaystyle\prod_{i=1}^{r}|StarFC(R_{M_{i}})|\displaystyle\prod_{M\in \mathcal{M}_{1}}|StarFC(R_{M})|=\\
\displaystyle\prod_{i=1}^{r}|StarFC(R_{M_{i}})|\leq |StarFC(R)|$ (since $\displaystyle\prod_{M\in \mathcal{M}_{1}}|StarFC(R_{M})|=1$ and $\displaystyle\prod_{i=1}^{r}|StarFC(R_{M_{i}})|\leq |StarFC(R)|$ by the proof of the claim). It follows that $|StarFC(R)|=\displaystyle\prod_{M\in Max(R)}|StarFC(R_{M})|$.
\end{proof}
In \cite[Theorem 1.5]{hmp4} it was proved that for a Noetherian domain $R$ with $Star(R)$ finite, $R$ is star regular if and only if $R_{M}$ is star regular for every maximal ideal $M$ of $R$. Our next theorem generalizes this result to one-dimensional quasi-Pr\"ufer domain such that $StarFC(R)$ is finite. Notice that, for a Noetherian domain with $Star(R)$ finite, $dim(R)=1$ and so $R'$ is a Dedekind domain and so $R$ is a quasi-Pr\"ufer domain.
\begin{thm}\label{Th3} Let $R$ be a one-dimensional quasi-Pr\"ufer domain such that $StarFC(R)$ is finite. Then $R$ is finitely star regular if and only if $R_{M}$ is finitely star regular for every maximal ideal $M$ of $R$.
\end{thm}
\begin{proof} Notice that each overring $T\subsetneqq qf(R)$ of $R$ is a one-dimensional quasi-Pr\"ufer domain.\\
$\Longleftarrow)$ Assume that $R_{M}$ is finitely star regular for every $M\in Max(R)$. Let $T$ be a proper overring of $R$. By Theorem~\ref{Th2}, $|StarFC(T)|=\displaystyle\prod_{N\in Max(T)}|StarFC(T_{N})|\leq \displaystyle\prod_{Q=N\cap R: N\in Max(T)}|StarFC(R_{Q})|\leq \displaystyle\prod_{M\in Max(R)}|StarFC(R_{M})|=|StarFC(R)|$. Thus $R$ is finitely star regular.\\
$\Longrightarrow)$ We mimic the proof of \cite[Theorem 1.5]{hmp4}. Assume that $R$ is finitely star regular and let $M$ be a maximal ideal of $R$ and suppose that there is an overring $T$ of $R_{M}$ such that $|StarFC(T)|>|StarFC(R_{M})|$. Set $B:=\displaystyle\bigcap_{Q\in Max(R)\setminus\{M\}} R_{Q}$ and $S=T \cap B$.  We first note that for every maximal ideal $N$ of $T$, $N\cap R_{M}=MR_{M}$ since $R_{M}$ is one-dimensional local domain. So $N\cap R=M$. But since $T$ is one-dimensional quasi-Pr\"ufer domain, $T$ is $h$-local and so $T$ has only finitely many maximal ideals, say $N_1, \ldots, N_r$.  Since $R$ is a one-dimensional quasi-Pr\"ufer domain, it is $h$-local, and hence $B_{R \setminus M} =K$ by \cite[Theorem 22]{m2}. This  yields $S_{R \setminus M} = T_{R \setminus M} \cap B_{R \setminus M}=T \cap K=T$.  It follows that $N_i \cap S \ne N_j \cap  S$ for $i \ne j$.  It is clear that if $Q$ is a maximal ideal of $S$ different from the $N_i \cap S$, then $Q \cap R=P$ for some maximal ideal $P$ of $R$ distinct from $M$ and hence $P=PR_ P\cap S$ and $S_Q=R_P$.  We then have $Max(S)=\{N_i \cap S \mid i=1, \ldots,n\} \cup \{PR_P \cap S \mid P\not= M\}$.  Therefore,
\begin{align} |StarFC(S)| &= \prod_{i=1}^n |StarFC(S_{Q_i \cap S})| \cdot \prod_{M \not=N} |StarFC(R_M)| \notag\\ &= \prod_{i=1}^n |StarFC(T_{Q_i})| \cdot \prod_{M \ne N} |StarFC(R_M)| \notag \\ &= |StarFC(T)| \cdot \prod_{M \ne N} |StarFC(R_M)| \notag \\ &> |StarFC(R_N)| \cdot \prod_{M \ne N} |StarFC(R_M)| \notag \\ &= |StarFC(R)|. \notag \end{align}  Therefore, $R$ is not star regular.
\end{proof}
\begin{corollary} Let $R$ be a domain with $StarFC(R)$ is finite and which satisfies one of the following conditions:\\
\1 Each proper overring of $R$ is Archimedean;\\
\2 Each proper valuation overring of $R$ satisfies the $accp$;\\
\3 Each proper overring of $R$ is a Mori domain.\\
Then $R$ is finitely star regular if and only if $R_{M}$ is finitely star regular for every maximal ideal $M$ of $R$.
\end{corollary}

\begin{proof} By \cite[Proposition 3.1(b), Lemma 3.2 and Corollary 3.9]{BarDob}, either $R$ is a valuation domain or $dim(R')=1$ and $R'$ is Pr\"ufer. If $R=V$ is a valuation domain, then $R$ is finitely star regular. Assume that $dim(R')=1$ and $R'$ is Pr\"ufer. Then $R$ has the finite character and for each overring $T$ of $R$, either $T$ is a valuation or $dim(T)=1$ and $T'$ is Pr\"ufer. If $T$ is valuation, $|StarFC(T)|=1\leq |StarFC(R)|$. Assume that
$dim(T)=1$ and $T'$ is Pr\"ufer. By Theorem~\ref{Th2}, $|StarFC(T)|=\displaystyle\prod_{N\in Max(T)}|StarFC(T_{N})|\leq \displaystyle\prod_{Q=N\cap R, N\in Max(T)}|StarFC(R_{Q})|\leq \displaystyle\prod_{M\in Max(R)}|StarFC(R_{M})|=|StarFC(R)|$ as desired.
\end{proof}

\begin{corollary}(\cite[Theorem 1.5]{hmp4}) Let $R$ be a Noetherian domain with $Star(R)$ finite.\\
Then $R$ is star regular if and only if $R_{M}$ is star regular for every maximal ideal $M$ of $R$. 
\end{corollary}

\begin{proof} Notice that $Star(R)=StarFC(R)$ and $dimR=1$. Since $R'=\overline{R}$ is a Krull domain, $R'$ is a Dedekind domain. The conclusion follows now from Theorem~\ref{Th3}.
\end{proof}

\section{pullback constructions}\label{PC}

Let $T$ be a domain, $M$ a maximal ideal of $T$, $K$ its residue field, $\phi:T\longrightarrow K$ the canonical surjection, $D$ a proper subring of $K$, and $k:=qf(D)$. Let $R$ be the pullback issued from the following diagram of canonical homomorphisms:
\[\begin{array}{cccl}
                    &R:=\phi^{-1}(D) & \longrightarrow                       & D\\
(\ \square\ )       &\downarrow         &                                       & \downarrow\\
                    &T                  & \stackrel{\phi}\longrightarrow     & K=T/M.
\end{array}\]
Clearly, $M=(R:T)$ and $D\cong R/M$. For ample details on the ideal structure of $R$ and its ring-theoretic properties, we refer the reader to \cite{ABDFK,AD,BG,BR,F,FG,GH}. The case where $T=V$ is a valuation domain is crucial and we will refer to this case as a classical diagram of type $(\square)$. Notice that for the classical diagram, if $I$ is a (fractional) ideal of $R$, then either $I=\phi^{-1}(E)$ for some $E\in F(D)$ if $M\subset I$, or $I$ is an ideal of $V$ or $I=a\phi^{-1}(E)$ for some $0\not=a\in M$ and $E$ a $D$-submodule of $K$ with $D\subseteq E\subset K$ if $I$ is not an ideal of $V$ (the proof similar to that of \cite[Theorem 2.1]{BG}). For more on star operations on pullbacks, see \cite{FP1, FP2}.
\begin{thm}\label{Pull.1} For the classical diagram of type$(\square)$, assume that $qf(D)=K$. Then $R$ is finitely star regular if and only if $D$ is finitely star regular.
\end{thm}

\begin{proof} Assume that $D$ is finitely star regular and let $T$ be an overring of $R$. If $V\subseteq T$, then $T$ is a valuation domain and so $|StarFC(T)|=1\leq |StarFC(R)|$. Assume that $R\subseteq T\subsetneqq V$. Then $T=\phi^{-1}(D_{1})$ where $D_{1}$ is an overring of $D$. Now by \cite[Theorem 4.4]{KadMi}, $|StarFC(T)|=|StarFC(D_{1})|\leq |StarFC(D)|=|StarFC(R)|$, and therefore $R$ is finitely star regular.\\
Conversely, assume that $R$ is finitely star regular and let $D_{1}$ be an overring of $D$. Set $T=\phi^{-1}(D_{1})$. Then $T$ is an overring of $R$ and again by \cite[Theorem 4.4]{KadMi}, $|StarFC(D_{1})|=|StarFC(T)|\leq |StarFC(R)|=|StarFC(D)|$, and therefore $D$ is finitely star regular.
\end{proof}
\begin{example} Let $k$ be a field , $X$ an indeterminate over $k$ and set $D=k[[X^{3}, X^{4}, X^{5}]]$, $V=k((X))[[Y]]=k((X))+M$ and $R=D+M$.
Clearly $R$ is neither integrally closed nor Noetherian domain. By \cite[Theorem 3.8]{hmp2}, $|StarFC(D)|=|Star(D)|=3$.
Since the only overrings of $D$ are $D_{1}=k[[X^{2}, X^{3}]]$, $D_{2}=k[[X]]$ and $qf(D)=k((X))$ and since $D_{1}$ and $D_{2}$ are Noetherian divisorial domains, $D$ is finitely star regular. Hence $R$ is finitely star regular by Theorem~\ref{Pull.1}. In fact the only proper overrings of $R$ are $T_{1}=D_{1}+M$, $T_{2}=D_{2}+M$ and $V$. By \cite[Theorem 4.4]{KadMi}, $|StarFC(T_{1})|=|StarFC(T_{2})|=|StarFC(V)|=1$ while $|StarFC(R)|=3$.
\end{example}

Our next theorem deals with an important of class of finitely star regular domains that are not in general star regular. It shows that any $PVD$ is a finitely star regular domain. Recall from Hedstrom and Houston (\cite{HH}) that a domain $R$ is
pseudo-valuation domain if it is quasilocal and shares its maximal ideal
with a valuation domain which necessarily must contain $R$ and be unique.
In terms of pullbacks, according to \cite[Proposition 2.6]{AD}, $R$ is a pseudo-valuation domain if and only if there
is a valuation domain $V$ with maximal ideal $M$ and a subfield $k$ of $V/M=K$
such that $R$ is the pullback in the following diagram
\[\begin{array}{ccl}
R            & \longrightarrow                 & k\\
\downarrow   &                                 & \downarrow\\
V            & \stackrel{\phi}\longrightarrow  & K=V/M
\end{array}\]
Notice that a $PVD$ which is not a valuation domain is a $TV$-domain, that is, the $t$- and $v$-operations are the same (\cite[Proposition 4.3]{HoZ}). We start with the following useful lemma.
\begin{lemma}\label{PVD.1} Let $R$ be a $PVD$ (which is not a valuation domain), $V$ its associated valuation overring, $M$ its maximal ideal, $k=R/M$ and $K=V/M$. If $StarFC(R)$ is finite, then $K$ is algebraic over $k$.
\end{lemma}

\begin{proof} Assume that $StarFC(R)$ is finite and suppose that $K$ is transcendental over $k$. Let $\lambda\in K$ transcendental over $k$ and set $T=\phi^{-1}(k[\lambda])$. Since $(k:k[\lambda])=(0)$, 
$T^{-1}=(R:T)=\phi^{-1}(k:k[\lambda])=\phi^{-1}(0)=M$, and so $T_{t}=T_{v}=V$. Now, for every nonzero prime ideals $p\not=q$ of $k[\lambda]$, set $T_{P}=\phi^{-1}(k[\lambda]_{P})$ and $T_{q}=\phi^{-1}(k[\lambda]_{q})$. By Lemma~\ref{L1}, $\delta(d_{T_{P}}, t)$ and $\delta(d_{T_{q}}, t)$ are star operations on $R$ of finite type and 
$T^{\delta(d_{T_{P}}, t)}=TT_{P}\cap T_{t}=T_{P}\cap V=T_{P}$ and $T^{\delta(d_{T_{q}}, t)}=TT_{q}\cap T_{t}=T_{q}\cap V=T_{q}$. Thus $\delta(d_{T_{P}}, t)\not=\delta(d_{T_{q}}, t)$. As $Spec((k[\lambda])$ is inifinte, $starFC(R)$ would be infinite, which is absurd. It follows that $K$ is algebraic over $k$.
\end{proof}
\begin{thm}\label{Pull.2} Any $PVD$ is finitely star regular.
\end{thm}

\begin{proof} First, if $R$ is a valuation domain, then for every overring $T$ of $R$, $|StarFC(T)|=|StarFC(R)|=1$. So, without loss of generality, we may assume that $R$ is not a valuation domain and $StarFC(R)$ is finite. Let $V$ be the associated valuation of $R$, $M$ its maximal ideal, $k=R/M$ and $K=V/M$. By Lemma~\ref{PVD.1}, $K$ is algebraic over $k$. Now, let $T$ be a proper overring of $R$. If $V\subseteq T$, then $T$ is a valuation domain and so $|StarFC(T)|=1\leq |StarFC(R)|$, as desired. Assume that $R\subsetneq T\subsetneq V$. Then $T=\phi^{-1}(F)$ where $F$ is a subfield of $K$ with $k\subsetneqq F\subsetneqq K$. We claim that $\delta(-, t): StarFC(T)\longrightarrow StarFC(R), *\mapsto \delta(*, t)$ is a one-to-one map. Indeed, let $*\not=*'\in StarFC(T)$ and let $A$ be a finitely generated integral ideal of $T$ such that $A^{*}\not=A^{*'}$. Necessarily $A\subsetneqq M$ and $A$ is not an ideal of $V$. Then $A=a\phi^{-1}(W)$ where $F\subsetneqq W\subsetneqq K$ is a finite dimensional $k$-subspace of $K$. Set $W=\displaystyle\sum_{i=1}^{i=r}F\lambda_{i}$ and let $H=\displaystyle\sum_{i=1}^{i=r}k\lambda_{i}$ and $B=a\phi^{-1}(H)$. Then $B$ is a finitely generated ideal of $R$, $BT=A$ and $B_{t}=B_{v}=A_{v_{T}}=aV$. Thus $B^{\delta(*, t)}=(BT)^{*}\cap B_{t}=A^{*}\cap A_{v_{T}}=A^{*}$. Similarly,$B^{\delta(*', t)}=(BT)^{*'}\cap B_{t}=A^{*'}\cap A_{v_{T}}=A^{*'}$. Hence $B^{\delta(*, t)}\not= B^{\delta(*', t)}$ and therefore $\delta(-, t)$ is one-to-one. It follows that $|StarFC(T)|\leq |StarFC(R)|$ and therefore $R$ is finitely star regular.\\
\end{proof}
Notice that a $PVD$ is not necessarily a star regular domain as shown by the following example.\\

\begin{example} Let $V$ be a valuation domain with a principal maximal ideal $M$ and a non-maximal prime ideal $P$ such that $P=P^{2}$. Suppose that $K=V/M$ is a quadratic extension of a field $k$ and let $R$ be the $PVD$ arising from the diagram:\\
\[\begin{array}{cccl}
                    &R:=\phi^{-1}(k) & \longrightarrow                       & k\\
(\ \square\ )       &\downarrow         &                                       & \downarrow\\
                    &V                  & \stackrel{\phi}\longrightarrow     & K=V/M.
\end{array}\]
Since $[K:k]=2$, $R$ is a divisorial domain and so $|Star(R)|=1$. However, $V_{P}$ is an overring of $R$ and $|Star(V_{P})|=2$. Hence $R$ is not star regular.
\end{example}
\begin{example} Let $k=\mathbb{Z}_{2}$ and $K$ an extension of $k$ with $[K:k]=4$ (for instance, let $x$ be a root of the irreducible polynomial $f(Y)=Y^{4}+Y^{3}+1\in k[Y]$ and $K=k(x)$). Let $X$ be an indeterminate over $k$ and set $V=K[[X]]=K+M$ and $R=k+M$. Let $T$ be a proper overring of $R$. If $V\subseteq T$, $T$ is a valuation domain and so $|StarFC(T)|=1\leq |StarFC(R)|$. If $R\subsetneqq T\subsetneqq V$, then $T=\phi^{-1}(F)$ where $k\subsetneqq F\subsetneqq K$ is a subfield of $K$. Necessarily $[K:F]=2$ and so $T$ is a divisorial $PVD$. Hence $|StarFC(T)|=|Star(T)|=1\leq |Star(R)|=|StarFC(R)|=9$ by \cite{Par1}.
\end{example}
Recall that for the general pullback of type $(\square)$, every star operation $*$ on $R$ induces a star operation $*_{\phi}$on $D$ defined by $J^{*_{\phi}}=\phi((\phi^{-1}(J))^{*})$ for every $J\in F(D)$ (\cite[Proposition 2.7 and Proposition 2.6]{FP1}. In this context, it is easy to check that if $*$ is of finite type on $R$, then $*_{\phi}$ on $D$ is of finite type on $D$.

\begin{thm} For the classical pullback diagram of type $(\square)$, let $T=\phi^{-1}(D_{1})$ be an overring of $R$. Then $|StarFC(T)|\leq |StarFC(D_{1})||StarFC(R)|$. In particular, if $|StarFC(D_{1})|=1$ for every $D_{1}\in [D, K]$, then $R$ is finitely star regular.
\end{thm}

\begin{proof} Set $T=\phi^{-1}(D_{1})$ and consider the map $\delta: StarFC(T)\longrightarrow StarFC(D_{1})\times StarFC(R)$,\\
$*\mapsto (*_{\phi}, \delta (*, t))$. We claim that $\delta$ is one-to-one. Indeed, let $*\not=*'\in StarFC(T)$ and let $A$ be a finitely generated integral ideal of $T$ such that $A^{*}\not=A^{*'}$. Necessarily $A$ is not an ideal of $V$. If $M\subsetneqq A$, then $A=\phi^{-1}(J)$ for some finitely ideal $J$ of $D_{1}$. In this case $A^{*}=\phi^{-1}(J^{*_{\phi}})$ and $A^{*'}=\phi^{-1}(J^{*'_{\phi}})$. Thus $J^{*_{\phi}}\not=J^{*'_{\phi}}$ and so $*_{\phi}\not=*'_{\phi}$. Assume that $A\subsetneqq M$ and set $A=a\phi^{-1}(W)$ where $D_{1}\subsetneqq W\subsetneqq K$ is a finitely generated $D_{1}$-module. If $(D_{1}:W)\not=0$, then $W\subseteq qf(D_{1})$ and so $W$ would be a finitely generated fractional ideal of $D_{1}$. Thus $A^{*}=a\phi^{-1}(W^{*_{\phi}})$ and $A^{*'}=a\phi^{-1}(W^{*'_{\phi}})$. Thus $W^{*_{\phi}}\not=W^{*'_{\phi}}$ and so $*_{\phi}\not=*'_{\phi}$. Assume that $(D_{1}:W)=0$. Set $W=\displaystyle\sum_{i=1}^{i=r}D_{1}\lambda_{i}$ and let $H=\displaystyle\sum_{i=1}^{i=r}D\lambda_{i}$ and $B=a\phi^{-1}(H)$. Then $B$ is a finitely generated ideal of $R$, $BT=A$ and $B_{t}=B_{v}=A_{v_{T}}=aV$. Thus $B^{\delta(*, t)}=(BT)^{*}\cap B_{t}=A^{*}\cap A_{v_{T}}=A^{*}$. Similarly,$B^{\delta(*', t)}=(BT)^{*'}\cap B_{t}=A^{*'}\cap A_{v_{T}}=A^{*'}$. Hence $B^{\delta(*, t)}\not= B^{\delta(*', t)}$. Thus $\delta(*)\not=\delta(*')$ and hence $\delta$ is one-to-one. It follows that $|StarFC(T)|\leq |StarFC(D_{1}||StarFC(R)|$.\\
Now assume that $|StarFC(D_{1}|=1$ for every $D_{1}\in [D, K]$ and let $T$ be an overring of $R$. If $V\subseteq T$, then $|StarFC(T)|=1\leq |StarFC(R)$. Let $T$ be a proper overring of $R$. Assume that $R\subsetneqq T\subsetneq V$ and $T=\phi^{-1}(D_{1})$ where $D\subsetneqq D_{1}\subsetneqq K$. Then $|StarFC(T)|\leq |StarFC(D_{1})||StarFC(R)|=|StarFC(R)|$ and therefore $R$ is finitely star regular.
\end{proof}
\begin{example} Let $\mathbb{Q}$ be the field of rational numbers, and $X$ and $Y$ indeterminates over $\mathbb{Q}$. Set $D=\mathbb{Q}[[X^{2}, X^{3}]]$, $V=\mathbb{Q}(\sqrt{2})((X))[[Y]]=\mathbb{Q}(\sqrt{2})((X))+M$ and $R=D+M$. Clearly $[D, K]=\{D, \mathbb{Q}[[X]], \mathbb{Q}((X)), \mathbb{Q}(\sqrt{2})[[X^{2}, X^{3}]], \mathbb{Q}(\sqrt{2})[[X]], K\}$ and every $D_{1}\in [D, K]$ is divisorial. Thus $R$ is finitely star regular.
\end{example}


\end{document}